\documentclass[11pt, english]{article}
\usepackage{geometry}
\usepackage{amsthm}
\usepackage{amsmath}
\usepackage{amssymb}
\usepackage{graphicx}
\date{}
\makeatletter

\numberwithin{figure}{section}
\theoremstyle{plain}

\theoremstyle{plain}

\newtheorem{theorem}{Theorem}[section]
\newtheorem{predl}{Proposition}[section]
\newtheorem{opr}{Definition}[section]
\newtheorem{lm}{Lemma}[section]

\newtheorem{conj}{Conjecture}[section]
\numberwithin{equation}{section}

\title{Drinfeld Yangian of the queer Lie superalgebra. I}
\author{V. Stukopin}

\providecommand{\propositionname}{Proposition}
\providecommand{\theoremname}{Theorem}

\begin{document}

\maketitle

\begin{abstract}

Drinfeld Yangian of a queer Lie superalgebra is defined as the quantization of a Lie bisuperelgebra of twisted polynomial currents. An analogue of the new system of generators of Drinfeld is being constructed. It is proved for the partial case Lie superalgebra $sq_1$ that this so defined Yangian and the Yangian, introduced earlier by M. Nazarov using the Faddeev-Reshetikhin-Takhtadzhjan approach, are isomorphic.
\\

{\bf Mathematical Subject Classification (2000)}. Primary 17B37; Secondary 81R50, 13F60.
\\

{\bf Keywords} Super Yangian, Queer Lie Superalgebra, Drinfel'd Yangian, Hopf superalgebra.
\end{abstract}

\markright{Drinfeld Yangian of the queer Lie superalgebra. I}

\section{Introduction}

 The Yangians of simple Lie superalgebras were defined by Drinfeld in the mid-1980s (see, for example, \cite{Dr}, \cite{Dr1}, \cite{Dr2}, \cite{Dr3}, \cite{Ch-Pr}), as quantizations (or non-commutative deformations in the Hopf algebra class) of Lie polynomial Lie bisuperalgebras with a bracket, determined by a rational $r$-matrix (the Yang matrix). But the Yangian of general linear algebra appeared earlier in the works of the mathematical physicists of the Leningrad school L.D. Faddeev. In the second case, it was determined using the defining relations given by the Yang quantum $R$-matrix.  V.Drinfeld established the equivalence of these two different definitions of the Yangian for a special linear algebra, but he did not publish proof of this fact. The first published evidence appeared much later (first for $\mathfrak{sl}_2$-case and later, almost 20 years after the first works of Drinfeld and for the case $\mathfrak{sl}_n$ (see \cite{Mol})).

Later, in the early 90s of the last century, L.D. Faddeev, N.Yu. Reshetikhin and L. Takhtadzhjan published a detailed description of the approach to the definition of quantum groups based on the use of the quantum $R$- matrix for defining a system of defining relations (\cite{Mol}). Such an approach was naturally related to the quantum method of the inverse scattering problem developed by them, which is now more often called the algebraic Bethe ansatz. Because, it is natural to call their name the traditional definition for the mathematical physicists of the Yangian, in contrast to V. Drinfeld's definition.
 In the 90s of the last century, the Yangians of Lie superalgebras  began to be investigated (see the works \cite{N}, \cite{St}).
Moreover, in the first paper the definition of the Yangian of the general linear Lie superalgebra was given by the Faddeev-Reshetikhin-Takhdazhjan approach (the FRT approach), and in the second work, the Yangian of the special linear superalgebra was defined in accordance with Drinfeld's approach. The Yangian thus defined we call the Drinfeld Yangian. We note that the Yangians of Lie superalgebras have now found numerous applications in quantum field theory, in particular, in the quantum theory of superstrings (\cite{D-N-W}, \cite{S-T}).

For definition Yangian  in accordance with the Drinfeld approach it is necessary to have a nondegenerate invariant bilinear form on the original Lie algebra (or Lie superalgebra), which is used in the definition of the Young matrix or for definition cobracket in Lie bisuperalgebra. But such a form is absent in the queer Lie superalgebra, which prevents the extension of the Drinfeld definition of Yangian on  this case. Nevertheless, as will be seen below, instead of deforming the superalgebra of polynomial currents, one can consider the deformation of a twisted superalgebra of currents.
The Yangians associated with the queer Lie superalgebra (see \cite{K}, \cite{F-S}) were first constructed by M. Nazarov (\cite{N2}), using approach Faddeev-Reshetikhin-Takhdazhjan. Later in the paper by M. Nazarov and A. Sergeev (\cite{N1}) was given definition of Yangian of the queer Lie superalgebra using a general approach based on the centralizer construction of G. Olshansky. This definition was later used in \cite{PS} to establish the connection between the Yangian of queer Lie superalgebra and W-algebras, introduced in \cite{Pr}.

In the paper \cite{St21}, the Yangian of the queer Lie superalgebra was introduced, following the Drinfeld approach  as a quantization of the Lie bisuperalgebra of twisted currents. There was also obtained his description in terms of the current system of generators and defining relations (a new system of generators in the terminology of V. Drinfel'd, \cite{Dr2}). It should be noted that the Yangian of a queer Lie superalgebra is defined as the deformation of a Lie bisuperalgebra of twisted polynomial currents in which the structure of a Lie superalgebra is given by a rational r-matrix (see  \cite{St2}).

We note that this approach can not be fully extended to quantizations of other twisted superalgebras of currents, since, generally speaking, they can not be endowed with the structure of a Lie bisuperalgebra (see \cite{St3}). Nevertheless, with a small modification, it can be used for  definition of twisted Yangians, which, generally speaking, are not Hopf superalgebras.

In this paper we, which is the first part of the work devoted to the study of the Yangian of the queer Lie superalgebra, we define the Yangian of the queer Lie superalgebra based on the Drinfeld approach, and establish a connection with the  Yangian Nazarov of the queer Lie superalgebra in the particular case of the Lie superalgebra $q_1$. In the second part of the paper, we consider this connection in complete generality.
A few words about the organization of this work. In the second section we recall the definitions of the queer Lie superalgebra and the twisted current  Lie  bisuperelgebra. In the third section we describe the deformation of the twisted current Lie bisupergalgebra and introduce the main actor, the Yangian of the queer Lie superalgebra. In the third section we formulate the main result of the paper - the description of the Yangian in terms of the analogue of the new system of Drinfeld generators.  In the section 4 we recall the definition of the Yangian of the queer Lie superalgebras given by M. Nazarov and prove the theorem on the isomorphism of the Drinfeld Yangian of the queer Lie superalgebra and the Yangian introduced by M. Nazarov for the partial case $Q_1$. Our proof is based on well-known ideas and essentially uses the triangular decomposition of the transfer matrix used in definition of the defining relations in the Yangian introduced by M. Nazarov and the construction of the quasi-determinants that was introduced by I.M. Gelfand  and V.S. Retakh (see \cite{G-R1}). In the concluding section 5 we give idea of the proof of general case using quasi-determinant theory (see \cite{G-R1}, \cite{MR}) and  uses some features of the root system of the queer Lie superalgebra.

We note that simplifying the proof of the theorem on isomorphism from section 5, we somewhat changed the current system of generators in this paper in comparison with the papers \cite{St12}, \cite{St21}, \cite{St22} which led to a somewhat simpler and more natural form of defining relations.

We shall use the following standard notation. We denote by $K[u]$, $K[[u]]$ the ring of polynomials, respectively, of formal power series, with coefficients in the ring $K$. Let $\mathbb{C}$ be the field of complex numbers, $\mathbb{N}$ the set of natural numbers, $\mathbb{Z}$ the ring of integers, $\mathbb{Z}_ + $ be  the set of nonnegative numbers, $\hbar$  will always denote the deformation parameter. The sign $\square$  denotes the end of the proof.

\section{Twisted current Lie bisuperalgebra} \label{s2}

In this section we define the classical analogue of the Yangian of the queer Lie superalgebra - the current twisted Lie bisupergalgebra.

\subsection{Queer Lie superalgebra} \label{s21}

In this subsection we recall the definition of a queer Lie superalgebra (see more detailed exposition in the papers \cite{K}, \cite{F-S}).

Let  $\mathbb{C}(n|n) = \mathbb{C}^n \oplus \mathbb{C}^n$ be a $\mathbb{Z}_2$-graded vector space of dimension  $(n|n)$ over complex number field $\mathbb{C}$. Let \\ $(e_{-n}, e_{-n+1}, \ldots, e_{-1}, e_1, \ldots, e_n)$ be a standard base in  $\mathbb{C}(n|n)$, and  $End(\mathbb{C}(n|n))$ be a superalgebra of linear operators acting in  $\mathbb{C}(n|n)$.

The basis in $End(\mathbb{C}(n|n))$ form matrices  $E_{a,b}$, $-n \leq a, b \leq n$, $ab \neq 0$ and function of parity  $p$ for $E_{a,b}$ is defined by formula:
\begin{equation}
p(E_{a,b}) = |a| + |b|,
\end{equation}
where $|a| = p(a) = 0$,  if $a>0$ and $|a|=p(a)=1$, for $a<0$, $|a|, |b| \in \mathbb{Z}_2$.
I recall that general linear superalgebra Lie $\mathfrak{gl}(n|n)$ is defined as vector space (associative superalgebra) $End(\mathbb{C}(n|n))$ with (super)bracket  $[\cdot, \cdot]$,  given by formula:
\begin{equation}
[x, y] = x \cdot y - (-1)^{|x||y|} y \cdot x.
\end{equation}
Super trace on  $\mathfrak{gl}(n|n)$ is given by formula  $str(E_{ab}) = \delta_{ab}(-1)^{1 + |a|}$ on on the elements of the basis and extends linearly to all other elements of the vector space $\mathfrak{gl}(n|n)$.

Recall that a special linear Lie superalgebra  $\mathfrak{sl}(n|n)$ is defined as:
\begin{equation}
\mathfrak{sl}(n|n) = \{A \in \mathfrak{gl}(n|n)| str(A) = 0\}
\end{equation}

Let $$\sigma: \mathfrak{gl}(n|n) \rightarrow \mathfrak{gl}(n|n),$$
be an involutive automorphism given by the formula:
$$\sigma : E_{a,b} \mapsto E_{-a,-b}.$$
Note, that  $\sigma$ transfer $\mathfrak{sl}(n|n)$ into $\mathfrak{sl}(n|n)$. Let also
$$J = \sum_{i=1}^n (E_{i,-i} - E_{-i,i}).$$

\begin{opr}
A Lie superalgebra $q_n$ can be defined equivalently either as a centralizer of $J$ in $\mathfrak{gl}(n|n)$, or as the set of fixed points of an involutive automorphism $\sigma: \mathfrak{gl}(n|n) \rightarrow \mathfrak{gl}(n|n).$ Let also  $sq_n = [q_n, q_n].$
\end{opr}

We also note that $sq_n$ is the Lie subalgebra of the Lie superalgebra $q_n$, consisting of $2n \times 2n$ matrices from which the left lower and upper right blocks of $n \times n$ blocks have a zero trace, in contrast to diagonal blocks, not required to have a zero trace. Note that $sq_n$ contains the identity matrix.

A simple Lie superalgebra $Q_n$ is defined (see \cite{K}) as a quotient Lie superalgebra
$$Q_n = sq_n / CI_{2n},$$
which is also denoted by $psq_n.$

We choose the following basis in $q_n$
\begin{eqnarray}
&E^0_{a,b} = E_{a,b} + E_{-a,-b}, \quad \\
&E^1_{a,b} = E_{-a,b} + E_{a,-b}, \quad 1 \leq a,b \leq n. \quad
\end{eqnarray}

We note that the even part is isomorphic to the Lie algebra $\mathfrak{gl}_n$, and the odd part is also isomorphic to $\mathfrak{gl}_n$, considered as a $\mathfrak{gl}_n$-module relative to the adjoint action. At the same time, for $sq_n$ its even part is isomorphic to $\mathfrak{gl}_n$, and the odd part is isomorphic to $\mathfrak{sl}_n$, so that the following matrices form a basis there:
\begin{eqnarray}
&E^0_{a,b}, E^1_{a,b}, \quad 1 \leq a \neq b \leq n, \quad \\
&E^0_{a,a},  \quad  1 \leq a \leq n, \quad\\
&H^1_{a} = E^1_{a,a} - E^1_{a+1,a+1}, \quad 1 \leq a \leq n-1.
\end{eqnarray}

We also note that there exists a connection between $q_n$ and the (super) Hecke-Clifford algebras. We can describe $q_n$ as a special linear (super) algebra with values in the some Clifford algebra $C$, which is convenient for studying representations of the Lie superalgebra $q_n$. Namely, let $C = \left\langle 1, c \right \rangle$, where $c$ is an odd element, and $c^2 = 1$. Then the Lie superalgebra $\mathfrak{sl}_n(C)$ is isomorphic in the category of Lie superalgebras to $q_n$. This fact, despite its obviousness, is sometimes convenient to keep in mind.

Now we give a description of the strange Lie superalgebra $sq_n$, which is analogous to the description of a simple Lie algebra in terms of generators and Serre relations.

\begin{predl}
The Lie superalgebra $sq_n$ is isomorphic to the Lie superalgebra $\mathfrak {g}$ generated by the generators $x^{\pm}_{i,j}, h_{i,1}, 1\leq i \leq n-1, j=0,1$ and $h_{i,0}, 1 \leq i \leq n$  (generators with the second index $j = 1$ are odd) that satisfy the following system of defining relations:
\begin{eqnarray}
&[h_{i_1,1}, h_{i_2,1}] = 2 \delta_{i_1,i_2}(h_{i_1,0} + h_{i_1+1,0}) - 2\delta_{i_1+1,i_2}h_{i_2,0}, \quad i_1 \leq i_2, \nonumber\quad \\
&[h_{i_1,j_1}, h_{i_2,j_2}]=0, \quad (j_1,j_2) \neq(1,1), \label{12.1}\\
&[h_{i_1,0}, x^{\pm}_{i_2,j_2}]= \pm(\delta_{i_1,i_2} - \delta_{i_1,i_2+1})x^{\pm}_{i_2,j_2}, \quad\\
&[h_{i_1,1}, x^{\pm}_{i_2,j_2}]= (\pm 1)^{j_2+1}(2\delta_{i_1,i_2}\cdot\delta_{j_2,0} - (-1)^{j_2}\delta_{i_1,i_2+1}-\delta_{i_1+1,i_2})x^{\pm}_{i_2,j_2+1},\label{12.4} \quad\\
&[x^+_{i_1,j_1}, x^-_{i_2,j_2}] = \delta_{i_1,i_2}h_{i_1,j_1+j_2}, \quad j_1 \neq j_2, \quad \nonumber\\
&[x^+_{i_1,j_1}, x^-_{i_2,j_1}] = \delta_{i_1,i_2}(h_{i_1,0} - (-1)^{j_1}h_{i_1+1,0}), \quad \label{12.5}\\
&[x^{\pm}_{i_1,j_1}, x^{\pm}_{i_2,j_2}] = (\pm1)^{j_1+j_2+1}[x^{\pm}_{i_1,j_1+1}, x^{\pm}_{i_2,j_2+1}], i_1 \leq i_2; \quad \label{12.6}\\
&ad(x^{\pm}_{i_1,j_1})^2(x^{\pm}_{i_2,j_2})=0, \quad 1 \leq i_1, i_2 \leq n-1, j_1, j_2 = 0,1, \quad \nonumber \\
&|i_1 - i_2| \neq 1, j_1, j_2 = 0,1. \qquad \label{12.7}
\end{eqnarray}

\end{predl}

Consider the following generators of Lie superalgebras $gl(n,n)$:

$$h_{i,0,0} := E_{i,i} + E_{-i,-i}, \quad e_{i,0,0} := E_{i,i+1} + E_{-i,-i-1}, \quad f_{i,0,0} := E_{i+1,i} + E_{-i-1,-i},$$
$$h_{i,0,1} := E_{i,i} - E_{-i,-i}, \quad e_{i,0,1} := E_{i,i+1} - E_{-i,-i-1}, \quad f_{i,0,1} := E_{i+1,i} - E_{-i-1,-i},$$
$$h_{i,1,0} := E_{i,-i} + E_{-i,i}, \quad e_{i,1,0} := E_{i,-i-1} + E_{-i,i1}, \quad f_{i,1,0} := E_{i+1,-i} + E_{-i-1,i},$$
$$h_{i,1,1} := E_{i,-i} - E_{-i,i}, \quad e_{i,1,0} := E_{i,-i-1} - E_{-i,i1}, \quad f_{i,1,0} := E_{i+1,-i} - E_{-i-1,i}.$$

These generators, as is easily verified by direct computations, satisfy the following system of defining relations.

\begin{eqnarray}
&[h_{i,0,r}, h_{j,0,s}] = 0; \quad [h_{i,0,r}, h_{j,1,s}] = 0; \quad \\
&[h_{i,1,0}, h_{j,1,0}] = 2\delta_{i,j} h_{i,0,0}; \quad [h_{i,0,0}, h_{j,1,1}] = 0; \quad\\
& [h_{i,0,1}, h_{j,1,1}] = 2\delta_{i,j} h_{i,1,0}; \quad [h_{i,1,0}, h_{j,1,1}] = 0; \quad\\
&[h_{i,0,0}, e_{j,k,r}] = (\delta_{i,j} - \delta_{i-1,j}) e_{j,k,r}; \quad [h_{i,0,0}, f_{j,k,r}] = -(\delta_{i,j} - \delta_{i-1,j}) f_{j,k,r}, \quad  k, r = 0, 1; \quad\\
&[h_{i,0,1}, e_{j,1,r}] = (\delta_{i,j} + \delta_{i-1,j}) e_{j,1,r+1}; \quad [h_{i,0,1}, f_{j,1,r}] = (\delta_{i,j} + \delta_{i-1,j}) f_{j,1,r+1}, \quad r=0, 1; \quad\\
&[h_{i,1,0}, e_{j,0,0}] = (\delta_{i,j} - \delta_{i-1,j}) e_{j,1,0}; \quad [h_{i,1,0}, f_{j,0,0}] = -(\delta_{i,j} - \delta_{i-1,j}) f_{j,1,0}; \quad\\
&[h_{i,1,0}, e_{j,0,1}] = -(\delta_{i,j} + \delta_{i-1,j}) e_{j,1,1}; \quad [h_{i,1,0}, f_{j,0,1}] = -(\delta_{i,j} + \delta_{i-1,j}) f_{j,1,1}; \quad\\
&[h_{i,1,0}, e_{j,1,0}] = (\delta_{i,j} + \delta_{i-1,j}) e_{j,0,0}; \quad [h_{i,1,0}, f_{j,1,0}] = (\delta_{i,j} + \delta_{i-1,j}) f_{j,0,0}; \quad\\
&[h_{i,1,0}, e_{j,1,1}] = -(\delta_{i,j} - \delta_{i-1,j}) e_{j,0,1}; \quad [h_{i,1,0}, f_{j,1,1}] = (\delta_{i,j} - \delta_{i-1,j}) f_{j,0,1}; \quad\\
&[h_{i,1,1}, e_{j,1,0}] = (\delta_{i,j} + \delta_{i-1,j}) e_{j,0,1}; \quad [h_{i,1,1}, f_{j,1,1}] = (\delta_{i,j} + \delta_{i-1,j}) f_{j,0,1}; \quad\\
&[h_{i,1,1}, e_{j,1,1}] = -(\delta_{i,j} + \delta_{i-1,j}) e_{j,0,0}; \quad [h_{i,1,1}, f_{j,1,1}] = -(\delta_{i,j} + \delta_{i-1,j}) f_{j,0,0}; \quad\\
&[e_{i,0,r}, f_{j,0,s}] = \delta_{i,j}(h_{i,0, r+s} - h_{i+1,0, r+s}), \quad [e_{i,i,0}, f_{j,i+1,0}] = \delta_{i,j}(h_{i,1, 0} - h_{i+1,1, 0}), \quad\\
&[e_{i,1,0}, f_{j,1,0}] = \delta_{i,j}(h_{i,0, 0} + h_{i+1,0, 0}), \quad [e_{i,k,0}, f_{j,k+1,0}] = \delta_{i,j}(h_{i,1, 0} - h_{i+1,1, 0}), \quad\\
&[e_{i,0,1}, f_{j,1,1}] = \delta_{i,j}(h_{i,1, 0} + h_{i+1,1, 0}), \quad [e_{i,1,1}, f_{j,1,1}] = -\delta_{i,j}(h_{i,0, 0} + h_{i+1,0, 0}), \quad\\
&-[e_{i,1,0}, f_{j,1,1}] = [e_{i,1,1}, f_{j,1,0}] = \delta_{i,j}(h_{i,0,1} - h_{i+1,0,1}). \quad
\end{eqnarray}

\subsection{Twisted current Lie bisuperalgebra}\label{s22}

Let's consider current Lie superalgebra   $\mathfrak{gl}(n|n)\otimes \mathbb{C}[u^{\pm 1}]$. Let

\begin{equation}
L_{tw}q_n = \{X(u) \in \mathfrak{gl}(n|n)\otimes \mathbb{C}[u^{\pm 1}] | \tilde{\tau}(X(u)) = X(-u) \}.
\end{equation}
Here  $\tilde{\tau} = \sigma \otimes 1.$
Let's also
$$L'_{tw}q_n = [L_{tw}q_n,  L_{tw}q_n].$$

The superalgebra of currents described above will also be the main object of study in this section. But here it will be more convenient for us to slightly change the notation and the order of presentation.

Let   $\mathfrak{g} = A(n-1,n-1)$, $\sigma: \mathfrak{g} \rightarrow \mathfrak{g}$ be above defined  automorphism of second order, $\epsilon = -1, \mathfrak{g}^j = Ker(\sigma - \epsilon^jE), \mathfrak{g}=\mathfrak{g}^0 \oplus \mathfrak{g}^1$.
Let's extend  $\sigma$  to automorphism  $\tilde{\sigma}: \mathfrak{g}((u^{-1})) \rightarrow \mathfrak{g}((u^{-1}))$, of Laurent series with values in  $\mathfrak{g}$ by formula:
\begin{equation}
\tilde{\sigma}(x\cdot u^j)= \sigma(x)(-u)^j.
\end{equation}
Let's consider the following Manin triple  $(\mathfrak{P}, \mathfrak{P}_1, \mathfrak{P}_2)$:
$$(\mathfrak{P}=\mathfrak{g}((u^{-1}))^{\tilde{\sigma}}, \mathfrak{P}_1 = \mathfrak{g}[u]^{\tilde{\sigma}},
\mathfrak{P}_2=(u^{-1}\mathfrak{g}[[u^{-1}]])^{\tilde{\sigma}}.$$

Let's note, that $\mathfrak{g}[u]^{\tilde{\sigma}}= L'_{tw}q_n.$

Let's define the bilinear form  $\left\langle \cdot, \cdot \right\rangle$ on $\mathfrak{P}$ by formula:
\begin{equation}
\left\langle f, g \right\rangle=res(f(u),g(u))du,
\end{equation}
where $res(\sum_{k=-\infty}^n a_k \cdot u^k):= a_{-1}$, $(\cdot, \cdot)$ be an invariant bilinear form on $\mathfrak{g}$.

It is clear that $\mathfrak{P}_1, \mathfrak{P}_2$ are isotropic subsuperalgebras with respect to the form \\
$\left\langle \cdot, \cdot \right\rangle$. It is also easy to verify that the following decompositions hold:
\begin{equation} \label{equation31}
\mathfrak{g}[u]^{\tilde{\sigma}} = \bigoplus_{k=0}^{\infty}(\mathfrak{g}^0 \cdot u^{2k} \oplus \mathfrak{g}^1 \cdot u^{2k+1})
\end{equation}

\begin{equation}
\mathfrak{g}((u^{-1}))^{\tilde{\sigma}} = \bigoplus_{k \in Z}(\mathfrak{g}^0 \cdot u^{2k} \oplus \mathfrak{g}^1 \cdot u^{2k+1})
\end{equation}

We describe the structures of the Lie bisuperalgebra on $\mathfrak{g}[u]^{\tilde{\sigma}}$. Let $\{e_i\}$ be a basis in $\mathfrak{g}^0$ and $\{e^i\}$  be its dual relatively bilinear form  $(\cdot, \cdot)$ basis in $\mathfrak{g}^1$. Let  $\mathfrak{t}_0 = \sum e_i \otimes e^i, \mathfrak{t}_1 = \sum e^i \otimes e_i, \mathfrak{t} = \mathfrak{t}_0 + \mathfrak{t}_1$. Let's consider also the basis $\{e_{i, k} \}$ in $\mathfrak{P}_1$ and its dual $\left\langle\cdot, \cdot \right \rangle$ basis $\{e^{i, k}\} $ in $\mathfrak{P}_2$, which are defined by the following formulas:
\begin{eqnarray}
&e_{i,2k}=e_i \cdot u^{2k}, e_{i,2k+1}=e^i \cdot u^{2k+1}, k \in \mathbb{Z}_+,&\\
&e^{i,2k}=e^i \cdot u^{-2k-1}, e^{i,2k+1}=e_i \cdot u^{-2k-2}, k \in \mathbb{Z}_+.
\end{eqnarray}

We calculate  the canonical element $r$, which determines the commutator in $\mathfrak{P}$.
\begin{eqnarray*}
&r= \sum e_{i,k} \oplus e^{i,k} =
\sum_{k \in Z} \sum_i ( e_i \cdot v^{2k} \otimes e^i \cdot u^{-2k-1} + \\
&e^i \cdot v^{2k+1} \otimes e_i \cdot u^{-2k-2}) =
\sum_{k=0}^{\infty}((\sum e_i \otimes e^i) \cdot u^{-1}(\dfrac{v}{u})^{2k} + \\
&\sum_{k=0}^{\infty} ((\sum e^i \otimes e_i) \cdot u^{-1} (\dfrac{v}{u})^{2k+1}) =
\mathfrak{t}_0 \dfrac{u^{-1}}{(1 - (v/u)^2)} + \mathfrak{t}_1 \dfrac{u^{-1}(v/u)}{1 - (v/u)^2} = \\
&\dfrac{\mathfrak{t}_0 \cdot u}{(u^2 - v^2)} +
\dfrac{\mathfrak{t}_1 \cdot v}{u^2 - v^2} = \frac{1}{2} (\dfrac{1}{u-v} + \dfrac{1}{u+v}) \mathfrak{t}_0 +
(1/2) (\dfrac{1}{u-v} - \dfrac{1}{u+v}) \mathfrak{t}_1 = \\
&\dfrac{1}{2} \dfrac{\mathfrak{t}_0 + \mathfrak{t}_1}{u-v} + \frac{1}{2} \dfrac{\mathfrak{t}_0 - \mathfrak{t}_1}{u+v}=
\frac{1}{2} \sum_{k \in \mathbb{Z}_+} \dfrac{(\sigma^k \otimes id)(\mathfrak{t})}{u-\epsilon^k \cdot v}.%\)
\end{eqnarray*}
We denote by  $r_{\sigma}(u,v):=r$. Then the formula for the commutator $\delta$ takes the following form:
$$\delta : a(u) \rightarrow [a(u) \otimes 1 + 1 \otimes a(v), r_{\sigma}(u,v)].$$

\begin{predl}
The element $r_{\sigma}(v,u)$ has the following properties:\\
1) $r_{\sigma}(u,v)= -r_{\sigma}^{21}(u,v)$;\\
2)$[r_{\sigma}^{12}(u,v), r_{\sigma}^{13}(u,w)] + [r_{\sigma}^{12}(u,v), r_{\sigma}^{23}(v,w)]+
[r_{\sigma}^{13}(u,w), r_{\sigma}^{23}(v,w)] = 0.$\\
\end{predl}

\begin{proof}
Let's note that  $\mathfrak{t}_0^{21} = \mathfrak{t}_1, \mathfrak{t}_1^{21} = \mathfrak{t}_0$. Then \\
$r_{\sigma}^{21}(v,u)=
\frac{1}{2} \dfrac{\mathfrak{t}_1 + \mathfrak{t}_0}{v-u} + \frac{1}{2} \dfrac{\mathfrak{t}_1 - \mathfrak{t}_0}{v+u} = -r_{\sigma}(u,v)$.
Item 2) follows from the fact that the canonical element $r$ satisfies the classical Yang-Baxter equation (CYBE) $\left\langle r,r \right\rangle = 0$.
Actually, the function   $r(u,v)= \dfrac{\mathfrak{t}_0 + \mathfrak{t}_1}{u-v}= \dfrac{\mathfrak{t}}{u-v}$ satisfies a classical Yang-Baxter equation  (CYBE) (see \cite{Dr}):
\begin{equation} \label{equation21}
[r^{12}(u,v), r^{13}(u,w)] + [r^{12}(u,v), r^{23}(v,w)]+[r^{13}(u,w), r^{23}(v,w)] = 0
\end{equation}
Let  $s= - id$.
Apply the operator $id\otimes s^k \otimes s^l (k, l \in \mathbb{Z}_2)$  to the left side of equality (\ref{equation21}) and substitute $(-1)^k \cdot v, (-1)^l \cdot w $ instead of $v,  w$, respectively. Summing over $k, l \in \mathbb{Z}_2$ and using the fact that $(s \otimes s) (\mathfrak{t}) = - \mathfrak{t}$ we get that the left-hand side satisfies to item 2 of this proposition.
\end{proof}

%$\square$

We describe the twisted superalgebra $\mathfrak{g}[u]^{\tilde{\sigma}}$ in terms of generators and defining relations. We introduce generators by the following formulas:
\begin{equation}
H_{i,k,r} = h_{i,k,\bar{r}}, \quad E_{i,k,r} = e_{i,k,\bar{r}}, \quad F_{i,k,r} = f_{i,k,\bar{r}}, \quad
\end{equation}
where $i \in I={1,2,\ldots, n}$, $k = {0, 1}$, $r \in \mathbb{Z}_0$,  $\bar{r} \in \mathbb{Z}_2, \bar{r} \equiv r \ \mod \ 2$.

\begin{predl}
The Lie superalgebra $\mathfrak{g}[u]^{\tilde{\sigma}}$ is generated by the generators $H_{i,k,r}, E_{i,k,r},  F_{i,k,r}$, where $i \in I = {1,2, \ ldots, n} $, $k \in \{0, 1\}$, $r \in \mathbb{Z}_0$, which satisfy the following system of defining relations:
\begin{eqnarray}
&[H_{i,0,r}, H_{j,0,s}] = 0; \quad [H_{i,0,r}, H_{j,1,s}] = 0; \quad \\
&[H_{i,1, 2k}, H_{j,1,2s}] = 2\delta_{i,j} H_{i,0, 2(k+s)}; \quad [H_{i,0,2k}, H_{j,1,2s+1}] = 0; \quad\\
& [H_{i,0,2k+1}, H_{j,1,2s+1}] = 2\delta_{i,j} h_{i,1,2(k+s+1)}; \quad [H_{i,1,2k}, H_{j,1,2s+1}] = 0; \quad\\
&[H_{i,0,2s}, E_{j,k,r}] = (\delta_{i,j} - \delta_{i-1,j}) E_{j,k,2s+r}; \quad \\
&[H_{i,0,2s}, F_{j,k,r}] = -(\delta_{i,j} - \delta_{i-1,j}) F_{j,k,2s+r}, \quad  k,r \in \mathbb{Z}_{0}; \quad\\
&[H_{i,0,2s+1}, E_{j,1,r}] = (\delta_{i,j} + \delta_{i-1,j}) E_{j,1, 2s+r+1}; \quad\\
&[H_{i,0,2s+1}, F_{j,1,r}] = (\delta_{i,j} + \delta_{i-1,j}) F_{j,1,2s+r+1}, \quad k,r \in \mathbb{Z}_{0}; \quad\\
&[H_{i,1,2s}, E_{j,0,2r}] = (\delta_{i,j} - \delta_{i-1,j}) E_{j,1,2(s+r)}; \quad [H_{i,1,2s}, F_{j,0,2r}] = -(\delta_{i,j} - \delta_{i-1,j}) F_{j,1,2(s+r)}; \qquad\\
&[H_{i,1,2s}, E_{j,0,2r+1}] = -(\delta_{i,j} + \delta_{i-1,j}) E_{j,1,2(s+r)+1}; \quad [H_{i,1,2s}, F_{j,0,2r+1}] = -(\delta_{i,j} + \delta_{i-1,j}) F_{j,1,2(s+r)+1}; \qquad
\end{eqnarray}
\begin{eqnarray}
&[H_{i,1,2s}, E_{j,1,2r}] = (\delta_{i,j} + \delta_{i-1,j}) E_{j,0,2(s+r)}; \quad [H_{i,1,2s}, F_{j,1,2r}] = (\delta_{i,j} + \delta_{i-1,j}) F_{j,0,2(s+r)}; \qquad\\
&[H_{i,1,2s}, E_{j,1,2r+1}] = -(\delta_{i,j} - \delta_{i-1,j}) E_{j,0,2(s+r)+1}; \quad\\
&[H_{i,1,2s}, F_{j,1,2r+1}] = (\delta_{i,j} - \delta_{i-1,j}) F_{j,0,2(r+s)+1}; \quad\\
&[H_{i,1,2s+1}, E_{j,1,2r}] = (\delta_{i,j} + \delta_{i-1,j}) E_{j,0,2(r+s)+1}; \quad\\
&[H_{i,1,2s+1}, F_{j,1,2r}] = (\delta_{i,j} + \delta_{i-1,j}) F_{j,0,2(r+s)+1}; \quad\\
&[H_{i,1,2s+1}, E_{j,1,2r+1}] = -(\delta_{i,j} + \delta_{i-1,j}) E_{j,0,2(r+s+1)}; \quad\\
&[H_{i,1,2s+1}, F_{j,1,2r+1}] = -(\delta_{i,j} + \delta_{i-1,j}) F_{j,0,2(r+s+1)}; \quad\\
&[E_{i,0,r}, F_{j,0,s}] = \delta_{i,j}(H_{i,0, r+s} - H_{i+1,0, r+s}), \quad\\
&[E_{i,i,2s}, F_{j,i+1,2r}] = \delta_{i,j}(H_{i,1, 2(s+r)} - H_{i+1,1, 2(s+r)}), \quad\\
&[E_{i,1,2s}, F_{j,1,2r}] = \delta_{i,j}(H_{i,0, 2(r+s} + H_{i+1,0, 2(r+s)}), \quad\\
&[E_{i,k,2s}, F_{j,k+1,2r}] = \delta_{i,j}(H_{i,1, 2(k+s)} - H_{i+1,1, 2(k+s)}), \quad\\
&[E_{i,0,2s+1}, F_{j,1,2r+1}] = \delta_{i,j}(H_{i,1, 2(s+r+1)} + H_{i+1,1, 2(s+r+1)}), \quad\\
& [E_{i,1,2s+1}, F_{j,1,2r+1}] = -\delta_{i,j}(H_{i,0, 2(s+r+1)} + H_{i+1,0, 2(s+r+1)}), \quad\\
&-[E_{i,1,2s}, F_{j,1,2r+1}] = [E_{i,1,2s+1}, F_{j,1,2r}] = \delta_{i,j}(H_{i,0,2(s+r)+1} - h_{i+1,0,2(s+r)+1}). \quad
\end{eqnarray}
\end{predl}

\section{Quantization. Definition of the Drinfeld Yangian of the queer Lie superalgebra}

\subsection{Quantization}

The general definition  of the twisted Yangian is considered in \cite{St3}. The question of the connection between this construction and the construction of the Yangian of the queer Lie superalgebra is also considered there. In the case where the twisted superalgebras of currents is a Lie bisuperelgebra, we obtain the following definition of quantization, which is a particular case of a more general construction from the paper \cite{St3}.
\begin{opr}
By the quantization of the Lie bisuperalgebra $(B, \delta)$ we mean a Hopf superalgebra $A = A_{\hbar}$ such that \\
1) $A_{\hbar} / \hbar \cdot A_{\hbar} \cong UB$, as Hopf superalgebra ($UB$ is a universal enveloping siperalgebra of Lie superalgebra  $B$);\\
2) $A \cong UB[[\hbar ]]$ are isomorphic as topological  $\mathbb{C}[[\hbar ]]$-modules;\\
3)$\delta (x) = \hbar^{-1} \cdot (\Delta(a) - \Delta^{op}(a))\bmod
\hbar $, where  $a$ be the inverse image of $x$ in $A$, and $\Delta^{op} = \tau \circ \Delta $, $\Delta$ is the comultiplication in $A$, and $\tau(x \otimes y) = (-1)^{p(x)p(y)} y \otimes x$ is the permutation of the tensor factors.
\end{opr}

We now describe the quantization of the Lie bisuperelgebra $(\mathfrak{g}[u]^{\tilde{\sigma}}, \delta)$. We constrain the quantization by the following additional conditions.\\

4) $A$ be a graded superalgebra over graded ring  $\mathbb{C}[[\hbar]]$, $\deg \hbar = 1$;\\
5) grading of  $A$ and grading of $\mathfrak{g}[u]^{\tilde{\sigma}}$ (with degrees u) induce the same grading of $U(\mathfrak{g}[u]^{\tilde{\sigma}})$, that is  $A_\hbar / \hbar \cdot A_\hbar \cong U(\mathfrak{g}[u]^{\tilde {\sigma }})$, as graded superalgebras over  $\mathbb{C}$.

\begin{predl}
\label{proposition1}
Let $\mathfrak{A}$ be a Lie superalgebra with an invariant bilinear form $(\cdot, \cdot)$; $\{e_i \}, \{e ^ i\} $ are dual bases relative to this bilinear form. Then for each element $g \in \mathfrak{A}$ we have the equality:
\begin{equation}
[g \otimes 1, \sum e_i \otimes e^i] = -[1 \otimes g, \sum e_i \otimes e^i]
\end{equation}
\end{predl}
%{\it Доказательство.}
\begin{proof}
The following equality follows from the definition of an invariant bilinear form: \\
\( ([g, a], b)= -(-1)^{deg(g)deg(a)}([a,g],b)= -(-1)^{deg(g)deg(a)}(a,[g,b])\) for $\forall a, b \in \mathfrak{A}$.  Therefore,
\begin{equation}\label{equation32}
([g, e_i], e^i)= -(-1)^{deg(g)deg(a)}([e_i, g],e^i)= -(-1)^{deg(g)deg(a)}(e_i,[g,e^i])
\end{equation}

A scalar product on a vector space $V$ defines an isomorphism between $V$ and $V^*$, and hence between $V \otimes V$ and $V \otimes V^*$. Summarizing the equality \ref{equation32} for $i$, we obtain

\[ \sum_i ([g, e_i], e^i)= \sum_i -(-1)^{deg(g)deg(a)}(e_i,[g,e^i]).\]

We note that the equality of the values of the functionals on the elements of the basis implies the equality of the functionals themselves and, in view of the above isomorphism, we obtain the following equality:

\[ \sum_i [g, e_i]\otimes e^i= \sum_i -(-1)^{deg(g)deg(a)}e_i \otimes [g,e^i] \]

or  \([g \otimes 1, \sum e_i \otimes e^i] = -[1 \otimes g, \sum e_i \otimes e^i].\)
Proposition is proved.
\end{proof}

Let $h^1_{i,j}$, $x^{\pm,1}_{i,j}$ are the elements of $\mathfrak{gl}(n,n)_1$ dual to $h_{i,j}, x^{\pm}_{i,j} \in q_n$, respectively.

\begin{predl}
Let \(\mathfrak{A}= \mathfrak{A}^0 \oplus \mathfrak{A}^1\) be a Lie superalgebra with nondegenerate invariant scalar product, such that  \(\mathfrak{A}^0, \mathfrak{A}^1\) are isotropic subspaces, and \(\mathfrak{A}^0, \mathfrak{A}^1\) are nondegenerate paired, \(\mathfrak{A}^0\) be a subsuperalgebra Lie, \(\mathfrak{A}^1\) be a module over \(\mathfrak{A}^0\). Let also   \(\{e_i\}, \{e^i\} \) be dual bases in \(\mathfrak{A}^0, \mathfrak{A}^1\), respectively,  \(\mathfrak{t}_0 = \sum_i e_i \otimes e^i, \mathfrak{t}_1 = \sum_i e^i \otimes e_i\). Then for the all   \( a \in \mathfrak{A}^0, b \in \mathfrak{A}^1\) we have the following equalities:
\[ [a\otimes 1, \mathfrak{t}_0] = - [1 \otimes a, \mathfrak{t}_0]; \quad
[a\otimes 1, \mathfrak{t}_1] = - [1 \otimes a, \mathfrak{t}_1]; \]
\[ [b\otimes 1, \mathfrak{t}_0] = - [1 \otimes b, \mathfrak{t}_1]; \quad
[b\otimes 1, \mathfrak{t}_1] = - [1 \otimes b, \mathfrak{t}_0]. \]
\end{predl}

\begin{proof}
The proof is based on the properties of an invariant bilinear form.

\end{proof}

\vspace{0.5cm}

 Now, using  the previous proposition,  we calculate the value of  $\delta$  on $h^1_{i,0}\cdot u$.
\begin{eqnarray*}
& \delta(h^1_{i,0} \cdot u) = [h^1_{i,0} \cdot v \otimes 1 + 1 \otimes h^1_{i,0} \cdot u, \frac{1}{2} \dfrac{\mathfrak{t}_0 + \mathfrak{t}_1}{u - v} + \frac{1}{2} \dfrac{\mathfrak{t}_0 - \mathfrak{t}_1}{u + v}] = \\
&[h^1_{i,0} \cdot v \otimes 1 - h^1_{i,0} \cdot u \otimes 1, \frac{1}{2} \dfrac{\mathfrak{t}_0 + \mathfrak{t}_1}{u - v}] + [ h^1_{i,0} \cdot v \otimes 1 + h^1_{i,0}\cdot u \otimes 1, \frac{1}{2} \dfrac{\mathfrak{t}_0 - \mathfrak{t}_1}{u + v} ] = \\
&[h^1_{i,0} \otimes 1, \frac{1}{2} (\mathfrak{t}_0 + \mathfrak{t}_1)] +
[h^1_{i,0} \otimes 1, \frac{1}{2}(\mathfrak{t}_0 - \mathfrak{t}_1)] = [h^1_{i,0} \otimes 1, \mathfrak{t}_0]= -[ 1 \otimes h^1_{i,0}, \mathfrak{t}_1].
\end{eqnarray*}

Similarly, the values of the cocycle on other generators can be calculated.

\begin{eqnarray*}
& \delta(h^1_{i,1} \cdot u) = -[h^1_{i,1} \otimes 1, \mathfrak{t}_0]= [1 \otimes h^1_{i,1}, \mathfrak{t}_1];\\
& \delta(x^{\pm,1}_{i,0} \cdot u) = -[x^{\pm,1}_{i,0} \otimes 1, \mathfrak{t}_0]= [1 \otimes x^{\pm, 1}_{i,0}, \mathfrak{t}_1];\\
& \delta(x^{\pm, 1}_{i,1} \cdot u) = -[x^{\pm,1}_{i,1} \otimes 1, \mathfrak{t}_0]= [1 \otimes x^{\pm, 1}_{i,1}, \mathfrak{t}_1].\\
\end{eqnarray*}

   From the condition of homogeneity of quantization it follows that \\
\(\Delta(h_{i,0,1})= \Delta_0(h_{i,0,1}) + \hbar F(x_{-\alpha,0,0}\otimes x_{\alpha,0,0}, x_{-\alpha,1,0} \otimes x_{\alpha,1,0} + h_{i, 1, 0} \otimes h_{i,1,0}) \)\\

From the correspondence principle (item 5) of the quantization definition) it follows that

\[\hbar^{-1}(\Delta(h_{i,0,1})- \Delta^{op}(h_{i,0,1}) = F - \tau F = [h^1_{i,0}\otimes 1, \mathfrak{t}_1]. \]

%\vspace{0.5cm}

Let  \[\bar{t}_0 = \sum_{\alpha \in \Delta_+} x_{-\alpha,0,1} \otimes x_{\alpha,0,0} - x_{\alpha,1,1} \otimes x_{-\alpha,1,0},\]

\(\Delta_+ -\) be a set of positive roots of simple Lie algebra   \(A_{n-1}= \mathfrak{sl}(n).\) \\

\vspace{0.5cm}

Let's define  \(\Delta(h_{i,1,1})\) by formula
 \[\Delta(h_{i,1,1})= \Delta_0(h_{i,1,1}) + \hbar [1 \otimes h^1_{i,1}, \mathfrak{t}_0].\]  Let us verify that the correspondence principle is satisfied in this case:

$$\hbar^{-1}(\Delta(h_{i,1,1})- \Delta^{op}(h_{i,1,1}) = [h^1_{i,1} \otimes 1, \sum_{\alpha \in \Delta_+} x^1_{\alpha,0} \otimes x^{-\alpha,0} - x^1_{\alpha,1} \otimes x^1_{-\alpha,1}]$$
$$- [1 \otimes h^1_{i,1} , \sum_{\alpha \in \Delta_+} x{-\alpha,0} \otimes x^1_{\alpha,0} - x_{-\alpha,1} \otimes x^1_{\alpha,1}].$$

We show, that
\[[h^1_{i,1} \otimes 1, \sum x^1_{-\alpha,0}\otimes x_{\alpha,0} + x^1{-\alpha,1}\otimes x_{\alpha,1}] = [1 \otimes h^1{i,1}, \sum x_{-\alpha,0}\otimes x^1_{\alpha,0} + x_{-\alpha,1}\otimes x^1_{\alpha,1}]. \]
Actually,
\[[h^1_{i,0} \otimes 1, \sum x^1_{-\alpha,0}\otimes x_{\alpha,0} + x^1_{-\alpha,1}\otimes x_{\alpha,1}] = \sum [h^1_{i,1}, x^1{-\alpha,0}]\otimes x_{\alpha,0} + [h^1_{i,0}, x^1_{-\alpha,1}] \otimes x_{\alpha,1}.\]
On the other hand \\
\[ [1 \otimes h^1_{i,1}, \sum x_{-\alpha,0}\otimes x^1_{\alpha,0} + x_{-\alpha,1}\otimes x^1_{\alpha,1}] = \sum x_{-\alpha,0}\otimes [h^1_{i,1}, x^1_{\alpha,0}] + x_{-\alpha,1} \otimes [h^1_{i,1}, x^1_{\alpha,1}].\]

Easy to check, that
$$[h^1_{i,1}, x^1_{\alpha_i - \alpha_j, 0}] = -(\delta_{ik} + \delta_{jk} - \delta_{i,k+1} - \delta_{j,k+1})x^1_{\alpha_i - \alpha_j,1}, \quad
[h^1_{i,1}, x^1_{\alpha_i - \alpha_j,0}] = (\delta_{ik} - \delta_{jk} - \delta_{i,k+1} + \delta_{j,k+1})x_{\alpha_i - \alpha_j,0}.$$
Therefore,
\[[h^1_{i,1}, x^1_{-\alpha,0}]\otimes x_{\alpha,0}= - x_{-\alpha,1} \otimes [h^1_{i,1}, x^1_{\alpha,1}], \quad
[h^1_{i,1}, x^1_{-\alpha,1}]\otimes x_{\alpha,1} = x_{-\alpha,0} \otimes [h^1_{i,1}, x^1_{\alpha,0}].\]

Then, \\
\[[h^1_{i,1} \otimes 1, \sum x^{-\alpha}\otimes x_{\alpha,0} + x^1{-\alpha,1}\otimes x_{\alpha,1}] = [1 \otimes h^1_{i,1}, \sum x_{-\alpha}\otimes x^1_{\alpha,0} + x_{-\alpha,1}\otimes x^1_{\alpha,1}] \]
and equality \\
\[h^{-1}(\Delta(h_{i,1,1})- \Delta^{op}(h_{i,1,1}) = -[1 \otimes h^1_{i,1}, \sum_{\alpha \in \Delta_+} x_{\alpha,0} \otimes x^1_{-\alpha,0} - x_{\alpha, 1} \otimes x^1_{-\alpha,1}] +\]
\[[h^1_{i,1} \otimes 1, \sum_{\alpha \in \Delta_+} x^1_{-\alpha,0} \otimes x_{\alpha,0} - x^1_{-\alpha,1} \otimes x_{\alpha,1}].\]
is proved.

Similarly, one can define a comultiplication on other generators. We get that
\[\Delta(x^+_{i,1,0})= \Delta_0(x^+_{i,1,0}) + \hbar [x^{+,1}_{i,0} \otimes 1, \bar{\mathfrak{t}}_0], \qquad
\Delta(x^-_{i,1,0}) = \Delta_0(x^-_{i,1,0}) + \hbar [1 \otimes x^{-,1}_{i,0} , \bar{\mathfrak{t}}_0].\]

It can be checked  directly that the following relations are preserved by comultiplication.
\[[h_{i,1,0}, x^{\pm}_{j,0,0}] = \pm (\delta_{i,j} - \delta_{i-1,j}) x^{\pm}_{j,1,0}, \qquad
[h_{i,1,1}, x^{\pm}_{j,0,0}] = \pm(a_{ij}) x^{\pm}_{j,1,1}, \]
\[ [h_{i,1,0}, x^{\pm}_{j,1,0}] = \pm (\delta_{i,j} - \delta_{i-1,j})x^{\pm}_{j,1,1}, \qquad
[h_{i,1,1}, x^{\pm}_{j,1,0}] = \pm \tilde{a}_{ij} x^{\pm}_{j,0,1}. \]

\vspace{0.5cm}

\[ [h_{i,1,0}, x^{\pm}_{j,0,0}]= \pm \tilde{a}_{ij} x^{\pm}_{j,1,0}, \qquad
[h_{i,1,1}, x^{\pm}_{j,0,0}]= \pm \tilde{a}_{ij} x^{\pm}_{j,1,1}, \]
\[[h_{i,1,0}, x^{\pm}_{j,1,0}]= \pm \tilde{a}_{ij} x^{\pm}_{j,1,1}, \qquad
[h_{i,1,1}, x^{\pm}_{j,1,0}]= \pm \tilde{a}_{ij} x^{\pm}_{j,0,1}. \]

\section{Yangian $Y(q_1)$}

In this section, we consider in detail the definition of the Yangian of the queer Lie superalgebra in the simplest case of a queer Lie superalgebra of type $q_1$, in order to demonstrate the features of our approach in the simplest case and compare our definition with M. Nazarov's approach. The main result of this section is the proof of the theorem on the isomorphism of the Yangian, obtained with our approach and the Yangian, defined by M. Nazarov.

\subsection{Queer Lie superalgebra $q_1$ and graded by involution Lie superalgebra $\mathfrak{gl}(1,1)$}

In this case we are dealing not with the basic Lie superalgebra. Nevertheless, the general construction of quantization considered in the previous subsection can also be used in this particular case. We recall the definition of the Lie superalgebra $q_1$. This superalgebra is generated by the generators $h_0, h_1$, and $p(h_0) = 0 $, $p (h_1) = 1$, which satisfy the following relations

\begin{eqnarray}
&[h_0, h_0]=0, \quad [h_0, h_1]=0, \quad\\
&[h_1, h_1] = 2h_0. \quad
\end{eqnarray}

The relations for a twisted superalgebra of currents with values in $\mathfrak{gl}(1,1)$ have the following form.

\begin{eqnarray}
&[h_0, h^1_0]=[h^1_0, h^1_0] = 0, \quad \\
&[h^1_0, h^1_1] = -[h^1_1, h^1_0] = \frac{1}{2}h_1, \quad\\
& [h^1_1, h^1_1] = -\frac{1}{2} h_0, \quad \\
&[h^1_0, h_1] = 2 h^1_1, \quad\\
&[h^1_1, h_1] = 2h^1_1. \quad
\end{eqnarray}

The Casimir element corresponding to the invariant scalar product in $\mathfrak{gl}(1,1)$ is defined by the formula:

\begin{equation}
\mathfrak{t}_0 = h^1_1 \otimes h_1 + h^1_0 \otimes h_0 - h_1 \otimes h^1_1 + h_0 \otimes h^1_0.
\end{equation}

We note that in this case the structure of the bisuperalgebra on the twisted current algebra $\mathfrak{gl}(1,1)^{tw}[t]$ is defined by a bracket
$$\delta: \mathfrak{gl}(1, 1)^{tw}[t] \rightarrow \mathfrak{gl}(1, 1)^{tw}[t] \hat{\otimes} \mathfrak{gl}(1, 1)^{tw}[t],$$
which is given by formula:

\begin{equation}
\delta: a(u) \rightarrow [a(u) \otimes 1 + 1 \otimes a(v), r_{\sigma}(u, v)], \quad r_{\sigma}(u, v) = \frac{1}{2}\dfrac{\mathfrak{t}_0 + \mathfrak{t}_1}{u - v} + \frac{1}{2}\dfrac{\mathfrak{t}_0 - \mathfrak{t}_1}{u + v}, \quad
\end{equation}
where

$$ \mathfrak{t}_0 = h_0 \otimes h^1_0 - h_1 \otimes h^1_1, \quad \mathfrak{t}_1 = h^1_0 \otimes h_0 - h^1_1 \otimes h_1.$$

As above, properties of $r_{\sigma}(u, v)$ are checked.

Next we will be interested in the Yangian of the Lie superalgebra $sq_1$, that is, $Y(sq_1) $, and also Yangian $Y (Q_1)$.

\begin{opr}
Yangian $Y(sq_1)$ is an associative superalgebra generated by generators $h_{0,0}, h_{1,0}, h_{1,0}, h_{1,1}$ which satisfy the following defining relations:
\begin{eqnarray}
&[h_{0,i}, h_{0,j}] = 0, \quad i,j \in \{0, 1\}, \qquad \\
&[h_{1,0}, h_{1,0}] = 2h_{0,0}, \qquad\\
&[h_{0,0}, h_{i,k}] = 0, \quad  i,j \in \{0, 1\}, \qquad \\
&[h_{0,1}, h_{1,0}] = 2 h_{1,1}, \qquad\\
&[h_{1,0}, h_{1,1}] = 0. \qquad
\end{eqnarray}

\end{opr}

Yangian $Y(sq_1)$ is a Hopf superalgebra with comultiplication on generators which is given by the following formulas:
\begin{eqnarray}
&\Delta(h_{i,0}) = h_{i,0}\otimes 1 + 1 \otimes h_{i,0}, \quad i = 0,1, \qquad\\
&\Delta(h_{0,1}) = h_{0,1}\otimes 1 + 1 \otimes h_{0,1} + h_{1,0} \otimes h_{1,0}, \qquad \\
&\Delta(h_{1,1}) = h_{1,1}\otimes 1 + 1 \otimes h_{1,1} - h_{1,0} \otimes h_{0,0} + h_{0,0} \otimes h_{1,0}. \qquad
\end{eqnarray}

\vspace{1cm}

\subsection{Quantization and definition of Yangian  $Y(q_1)$}

 We will use the modification of Drinfeld's approach to the definition of the Yangian, first proposed by S. Levendorskii.

Let's define a comultiplication on generators  $h_{0,1}$ and $h_{1,1}$ by the following formulas:

\begin{eqnarray}
&\Delta(h_{0,1}) = h_{0,1} \otimes 1 + 1 \otimes h_{0,1} + \hbar [2h^1_0 \otimes 1, h^1_0 \otimes h_0 - h^1_1 \otimes h_1], \quad \\
&\Delta(h_{1,1}) = h_{1,1} \otimes 1 + 1 \otimes h_{1,1} + \hbar [2h^1_1 \otimes 1, h^1_0 \otimes h_0 - h^1_1 \otimes h_1]. \quad
\end{eqnarray}

In other words we have the following equalities

\begin{eqnarray}
&\Delta(h_{0,1}) = h_{0,1} \otimes 1 + 1 \otimes h_{0,1}  - \hbar(h_{1,0} \otimes h_{1,0}), \quad \\
&\Delta(h_{1,1}) = h_{1,1} \otimes 1 + 1 \otimes h_{1,1} + \hbar(- h_{1,0} \otimes h_{0,0} + h_{0,1} \otimes h_{1,0}). \quad
\end{eqnarray}

Let's check that the comultiplication defined on generators of the first order satisfies the correspondence principle:

$$\hbar^{-1}(\Delta(h_{i,1}) - \Delta^{op}(h_{i,1})) = \delta(h^1_i \cdot u), \quad i = 0, 1.$$

In other words, we need to verify that equality
$$ \hbar^{-1}(\Delta(h_{i,1}) - \Delta^{op}(h_{i,1})) = [h^1_i \cdot u \otimes 1 + 1 \otimes h^1_i \cdot v, r_{\sigma}(u, v)].  $$

Let's note, that the proof of this fact repeats the reasoning in section 3.

\vspace{0.5cm}

\subsection{Nazarov Yangian  $Y_N(q_1)$}

Recall that Yangian $Y_N(q_1)$ introduced by M. Nazarov is the associative unital superalgebra over $\mathbb{C}$ with countable set of generators
$$t^m_{i,j}, \quad i,j = \pm 1, \quad m = 1,2, \ldots.  $$
The $\mathbb{Z}_2$ -- grading of the $Y_N(q_1)$ is defined as follows $p(t^m_{1,1}) = p(t^m_{-1,-1}) = 0$,  $p(t^m_{1,-1}) = p(t^m_{-1,1}) = 1$.

To write down the defining relations for these generators we employ the formal series in $Y_N(q_1)[[u^{-1}]]$:
$$ t_{i,j}(u) = \delta_{i,j}\cdot 1 + t^1_{i,j} u^{-1} + t^2_{i,j} u^{-2} + \ldots .   $$

Then for all possible indices $i,j,k,l$ we have the relations:

\begin{eqnarray} \label{eq:nq1}
&(u^2 - v^2)([t_{i,j}(u), t_{k,l}(v)]) \cdot (-1)^{p(i)p(k)+p(i)p(l)+p(k)p(l)} = \qquad \nonumber \\
&(u+v)(t_{k,j}(u)t_{i,l}(v) -t_{k,j}(v)t_{i,l}(u)) - (u-v)(t_{-k,j}(u)t_{-i,l}(v) -t_{k,-j}(v)t_{i,-l}(u))(-1)^{p(k)+p(l)}, \qquad
\end{eqnarray}
where $v$ is a formal parameter independent for $u$, and we have an equality in the algebra of formal Laurent series in $u^{-1}$, $v^{-1}$ with coefficients in $Y(q_1)$. Also we have the relations
\begin{equation} \label{eq:nq11}
t_{i,j}(-u) = t_{-i,-j}(u).
\end{equation}

Note, that the relations (\ref{eq:nq1}) and (\ref{eq:nq11}) are equivalent to the following defining relations:
\begin{eqnarray}
&([t^{m+1}_{i,j}, t^{r-1}_{k,l}] - [t^{m-1}_{i,j}, t^{r+1}_{k,l}])\cdot (-1)^{p(i)p(k)+p(i)p(l)+p(k)p(l)} = \quad \nonumber \\
&t^{m}_{k,j} t^{r-1}_{i,l} + t^{m-1}_{k,j} t^{r}_{i,l} - t^{r-1}_{k,j} t^{m}_{i,l} - t^{r}_{k,j} t^{m-1}_{i,l} + \quad \nonumber \\
&(-1)^{p(k)+p(l)}(-t^{m}_{-k,j} t^{r-1}_{-i,l} + t^{m-1}_{-k,j} t^{r}_{-i,l} + t^{r-1}_{k,-j} t^{m}_{i,-l} - t^{r}_{k,-j} t^{m-1}_{i,-l}), \qquad,  \label{eq:nq12}
\end{eqnarray}
\begin{equation} \label{eq:nq13}
t^m_{-i,-j} = (-1)^m t^m_{i,j},
\end{equation}
where $m,r = 1,2, \ldots$ and $t^0_{i,j} = \delta_{i,j}$.

Recall that $Y_N(q_1)$ is a Hopf algebra with comultiplication given by the formula
\begin{equation} \label{eq:comny1}
\Delta(t^m_{i,j}) = \sum_{s=0}^{m} \sum_{k} (-1)^{(p(i)+p(k))(p(j)+p(k))} t^s_{i,k} \otimes t^{m-s}_{k,j}.
\end{equation}

Let's note that this definition can be rewrite using generating functions as follows. Let, $I=\{1,2 \ldots, n$, $I_1 = I \bigcup (-I) = \{\pm 1, \pm 2, \ldots, \pm n \}$ and  as above
$$ J = \sum_{i \in I_1} E_{i,-i}(-1)^{p(i)}, $$
$$ P = \sum_{i, j \in I_1} E_{i,j}\otimes E_{j,i}(-1)^{p(j)},  \quad P_1 = P\otimes 1, \quad P_2 = 1 \otimes P,  $$
$$T(u) = \sum_{i,j \in I_1} E_{i,j} \otimes t_{i,j}(u), \quad T_1(u) = T(u) \otimes 1, \quad T_2(u) = 1 \otimes T(u). $$
We define quantum R-matrix by formula:
\begin{equation}
R(u,v) = 1 - \dfrac{P}{u-v} + \dfrac{PJ_1J_2}{u+v}.
\end{equation}

Then defining relations can be presented in the following form

\begin{equation}
(R(u,v)\otimes 1) T_(u)T_2(v) = T_2(v)T_1(u) (R(u,v)\otimes 1).
\end{equation}

We write explicitly the relations for the Yangian $Y_N(q_1)$:

\begin{equation*}
[t^m_{-1,1}, t^k_{-1,1}] = \sum_{r=1}^{m-1} \left(t^{k+r-1}_{-1,1} t^{m-r}_{-1,1} - t^{m-r}_{-1,1} t^{k+r-1}_{-1,1}\right) +
\end{equation*}

\begin{equation}
\sum_{r=1}^{m-1}(-1)^r \left((-1)^{k+m}t^{k+r-1}_{1,1} t^{m-r}_{1,1} - t^{m-r}_{1,1} t^{k+r-1}_{1,1}\right).
\end{equation}

Similarly, we obtain, that

\begin{equation*}
[t^m_{1,1}, t^k_{1,1}] = -\sum_{r=1}^{m-1} \left[t^{k+r-1}_{1,1}, t^{m-r}_{1,1}\right] +
\end{equation*}
\begin{equation}
\sum_{r=1}^{m-1}(-1)^r \left((-1)^{k+m}t^{k+r-1}_{-1,1} t^{m-r}_{-1,1} - t^{m-r}_{-1,1} t^{k+r-1}_{-1,1}\right).
\end{equation}

Last relation has the following form:

\begin{equation*}
[t^m_{1,1}, t^k_{-1,1}] = \sum_{r=1}^{m-1} \left(t^{k+r-1}_{1,1} t^{m-r}_{-1,1} - t^{m-r}_{1,1} t^{k+r-1}_{1,1}\right) +
\end{equation*}

\begin{equation}
\sum_{r=1}^{m-1}(-1)^r \left((-1)^{k+m}t^{k+r-1}_{-1,1} t^{m-r}_{1,1} - t^{m-r}_{-1,1} t^{k+r-1}_{1,1}\right).
\end{equation}

\vspace{0.5cm}

\subsection{Yangian  $Y(sq_1)$. Current generators.}

Let
$$\{a,b\} := a\cdot b + (-1)^{p(a)p(b)} b \cdot a $$
be an anticommutator.

\begin{opr}
Yangian  $\bar{Y}(sq_1)$ is  the associative unital superalgebra over $\mathbb{C}$ with countable set of generators $\tilde{h}_{0,k}$, $\tilde{h}_{1,k}$, $k \in \mathbb{Z}_{k \geq 0}$, which satisfy the following system of defining relations.

\begin{eqnarray}
&[\tilde{h}_{0,2k}, \tilde{h}_{0,2r}] = [\tilde{h}_{2k+1}, \tilde{h}_{2r+1}] = 0, \quad [\tilde{h}_{0,0}, \tilde{h}_{0,k}] =0,  \qquad, \\
&[\tilde{h}_{0, 1}, \tilde{h}_{1,k}] = 2\tilde{h}_{1,k} + \frac{1}{2}(\tilde{h}_{0,0}\tilde{h}_{1,k} + \tilde{h}_{1,k}\tilde{h}_{0,0}), \qquad \\
&[\tilde{h}_{1,0}, \tilde{h}_{1,k}] = \tilde{h}_{0,k}, \qquad, \\
&[\tilde{h}_{0, k+1}, \tilde{h}_{1,r}] = [\tilde{h}_{0,k}, \tilde{h}_{1, r+1}] + \{\tilde{h}_{0,k}, \tilde{h}_{1,r}\} + \{\tilde{h}_{1,k}, \tilde{h}_{0,r}\}, \qquad\\
& [\tilde{h}_{1, k+1}, \tilde{h}_{1,r}] - [\tilde{h}_{1,k}, \tilde{h}_{1, r+1}] = \{\tilde{h}_{1,k}, \tilde{h}_{1,r}\} + \{\tilde{h}_{0,k}, \tilde{h}_{0,r}\}, \qquad
\end{eqnarray}
\end{opr}
\vspace{0.5cm}

\begin{theorem} \label{th:01}
Yangian $Y(sq_1)$ is isomorphic as an associative algebra to $\bar{Y}(sq_1)$.
\end{theorem}

\begin{proof}
Let's define map
\begin{equation}
G: Y(sq_1) \rightarrow \bar{Y}(sq_1)
\end{equation}
by the following formulas on generators
\begin{eqnarray}
&G(h_{i,0}) = \tilde{h}_{i,0}, \quad \label{eq:CurY1} \\
&G(h_{0,1}) = \tilde{h}_{0,1} + \frac{1}{2}\tilde{h}_{0,0}, \qquad  \label{eq:CurY2}\\
&G(h_{1,1}) = \tilde{h}_{1,1}. \qquad  \label{eq:CurY3}
\end{eqnarray}

We show that the mapping given by the formulas (\ref{eq:CurY1}) --  (\ref{eq:CurY2}) respects the defining relations of $Y(sq_1)$ and $\bar{Y}(sq_1)$.

Let's define the new generators in $Y(sq_1)$ by the following formulas:

\begin{equation}
h_{1,k+1} := [h_{0,1}, h_{1,k}], \quad h_{0,2k} := [h_{1,0}, h_{1,2k}].
\end{equation}

Easy to check that the following relations are satisfied:

\begin{equation}
[h_{1,k+2}, h_{0,2r}] = [h_{1,k}, h_{0,2(r+1)}] + \ldots
\end{equation}

\end{proof}

We define the system of generators and defining relations of Drinfeld Yangian of the Lie superalgebra $sq_1$.

\begin{theorem}
Yangian $Y_D(sq_1)$ is an associative Hopf superalgebra over $\mathbb{C}$ generated by generators $h_{0,2k}$, $h_1{1,k}$, $k \in \mathbb{Z}_{k \geq 0}$, which satisfy the following defining relations:
\begin{eqnarray}
&[h_{2k}, h_{0,2l}] = 0, \quad \\
&[h_{0, 2k+2}, h_{1, l}] = [h_{1, 2k+1}, h_{0, l+1}] + [h_{0, 2k}, h_{1,l+2}] +\{h_{0,2k, h_{1,l+1}}\} + (-1)^{l}\{h_{1,2k},h_{0,l}\}, \quad\\
&[h_{1, k+2}, h_{1,l}] = [h_{1,k}, h_{1,l+2}] = \{h_{1,k+1}, h_{1,l}\}   + \{h_{0,k}, h_{0,l+1}\}. \quad
\end{eqnarray}
\end{theorem}

\vspace{0.5cm}

\vspace{0.5cm}

\subsection{Isomorphism between Nazarov Yangian  $Y_N(q_1)$ and Drinfeld Yangian $Y(q_1)$}

I recall the triangular decomposition construction for Yangian $Y_N(q_1)$. Let $T(u) = (t_{i,j}(u))_{i,j=\pm 1}$. Then we have the following decomposition
\begin{equation}
\begin{pmatrix} t_{1,1}(u)&t_{1,-1}(u) \\ t_{-1,1}(u)& t_{-1,-1}(u) \end{pmatrix} = \begin{pmatrix} 1 & 0 \\ h_{-1}(u)& 0 \end{pmatrix} \cdot \begin{pmatrix} d_{1}(u)& 0 \\ 0 & d_{2}(u) \end{pmatrix} \cdot \begin{pmatrix} 1 &h_{1}(u) \\ 0& 1 \end{pmatrix}.
\end{equation}

So, we have the following equality

\begin{equation}
\begin{pmatrix} t_{1,1}(u)&t_{1,-1}(u) \\ t_{-1,1}(u)& t_{-1,-1}(u) \end{pmatrix} = \begin{pmatrix} d_{1}(u)&d_{1}(u)\cdot h_1(u) \\ h_{-1}(u)\cdot d_1(u)& d_2(u) + h_{-1}(u) d_1(u) h_{1}(u). \end{pmatrix}
\end{equation}

We can formulate the main result of this subsection.

\begin{theorem} \label{th:02}
The correspondence
\begin{equation}
t_{1,1}(u)  \rightarrow h_0(u), \quad t_{1,1}(u)^{-1}t_{-1,1}(u) \rightarrow h_1(u)
\end{equation}
define the epimorphism
\begin{equation} \label{eq:F}
F: Y_N(q_1) \rightarrow Y_D(sq_1),
\end{equation}
where $Ker F = \langle h_{0,3}, h_{0,5}, \ldots    \rangle$ the ideal generated by elements $h_{0,2k+1}$ for $k\geq 1$.

%$SY_N(q_1) = Y_N(q_1) / \left\langle  det_q(T(u)) = 1 \right\rangle$.

\end{theorem}

\begin{proof}

Let's check that map $F$ respects the defining relations of $Y_N(q_1)$ and $Y_D(sq_1)$.   We have that
$$t_{1,-1} = d_1(u)\cdot h_1(u).  $$
Therefore
$$ t_{1,-1} = \sum_{k=1}^{\infty} t_{1,-1}^k u^{-k} = \left( 1 + \sum_{k=0}^{\infty} h_{0,k} u^{-k-1} \right) \left(\sum_{r=0}^{\infty} h_{1,r}u^{-r-1} \right) = \sum_{n=0}^{\infty}\left(\sum_{k=-1}^{n-1} h_{0,k}\cdot h_{n-k}\right) u^{-n-1}.  $$
Here $h_{0,-1} = 1$. Therefore
$$t^{n+1}_{1,-1} = \sum_{k=-1}^{n-1} h_{0,k}\cdot h_{n-k} = - h_{1,n} + h_{0,0}h_{1,n-1} + \ldots + h_{0,n-1}h_{0,0}.   $$

Also, we obtain
$$t_{1,1} = 1 + \sum_{n=1}^{\infty} t^n_{1,1} u^{-n} = d_1(u) = h_0(u) = 1 + \sum_{n=0}^{\infty} h_{0,n} u^{-n-1}.  $$
Therefore
$$ t^n_{1,1} = h_{0, n-1}.  $$

\end{proof}

As a corollary of proved theorem we obtain a definition of Yangian $$Y_N(sq_1):= Y_N(q_1)/Ker F.$$

In Nazarov Yangian $Y_N(sq_1)$ we'll use only even generators $t_{11}^{2k}$, because only this generators corresponds to generators $h_{0,2k}$ in Drinfel'd Yangian $Y_D(sq_1)$. We have the following lemma is hold.

\begin{lm}
We have the following relation
\begin{equation} \label{eq:com_t}
[t_{11}^{2k}, t_{11}^{2r}] =0.
\end{equation}
for all $r,k >0$.
\end{lm}

\begin{proof}
For $r,k = 0,1$  the formula (\ref{eq:com_t}) is evidently satisfied. Easy to check also that $[t_{11}^{0}, t_{11}^{2r}] = 0$ and $[t_{11}^{0}, t_{-11}^{2r}] = 0$  for all $r \geq 0$.

We suggest that $[t_{11}^{2k}, t_{11}^{2r}] =0$ and let's prove  that $[t_{11}^{2(k+1)}, t_{11}^{2(r+1)}] = 0 $ is also true.
\end{proof}

\vspace{2cm}

\section{Drinfeld Yangian of the queer Lie superalgebra. New system of generators}

In this section we define the analogue of the new system generators which was defined by V. Drinfeld for Yangians of simple Lie algebras.  I recall that for the special case $sq_1$ Drinfeld Yangian was introduced in the previous section. Now we define the Drinfeld Yangians $Y_D(sq_n)$ and $Y_D(Q_n)$ of the queer Lie superalgebras $sq_n$ and $Q_n$.

\begin{theorem} \label{th5.101}
Yangian $Y_D(sq_{n-1})$ is generated by generators
$$\bar{h}_{1,j,m}, \quad \tilde{h}_{i,j,m} = \bar{h}_{i,j,m} - \bar{h}_{i+1,j,m}, \quad  x^{\pm}_{i,j,m},$$
$i \in I = \{1, \ldots, n-1\}, \quad  m \in \mathbb{Z}_+, \quad j \in \mathbb{Z}_2 = \{0, 1\},$

$$\tilde{h}_{i,m} = \tilde{h}_{i,0,m}, \quad k_{i,m}=h_{i,1,m}, \quad x^{\pm}_{i,m} = x^{\pm}_{i,0,m}, \quad \hat{x}^{\pm}_{i,m}=x^{\pm}_{i,1,m},$$ $i \in \{1, \ldots, n-1\}, m \in \mathbb{Z}_+,$
which satisfies the following system of defining relations:

\begin{eqnarray}
&[\tilde{h}_{i_1, 0, m_1}, \tilde{h}_{i_2, 0, m_2}] = 0, \quad \label{eq54296} \\
&[x^+_{i,0,m}, x^-_{j,s,0}]= \delta_{ij}\tilde{h}_{i,s, m}, \quad s \in \{0, 1\} \qquad \label{54297}\\
&[\tilde{h}_{i_1, 1, 2m_1+s}, \tilde{h}_{i_2, 0, 2m_2+s}]=0, \quad s \in \{0, 1\} \quad \label{eq54299} \\
&[\tilde{h}_{i_1, 1, 2m_1+s}, \tilde{h}_{i_2, 1, 2m_2+s}]= (-1)^s2\delta_{i_1, i_2}(\bar{h}_{i_1, 0, 2(m_1 + m_2+s)} + \quad \nonumber \\
&\bar{h}_{i_1+1, 0, 2(m_1 + m_2+s)}) + (-1)^s2(\delta_{i_1+1, i_2} + \delta_{i_1 - 1, i_2})\bar{h}_{i_2, 0, 2(m_1 + m_2+s)}, \quad s \in \{0, 1\} \quad \label{eq54300}\\
&[\tilde{h}_{i_1, 1, 2m_1+1}, \tilde{h}_{i_2, 0, 2m_2}]= [\tilde{h}_{i_1, 1, 2m_1}, \tilde{h}_{i_2, 0, 2m_2+1}]= \quad \nonumber \\
&2\delta_{i_1, i_2}(\bar{h}_{i_1, 1, 2(m_1 + m_2)+1} + \bar{h}_{i_1+1, 1, 2(m_1 + m_2)+1}) - 2(\delta_{i_1+1, i_2}+ \delta_{i_1 - 1, i_2})\bar{h}_{i_2, 0, 2(m_1 + m_2)+1}, \quad  \label{eq54301} \\
&[\bar{h}_{i_1, 1, 2m_1+1}, \bar{h}_{i_2, 1, 2m_2}] = [\bar{h}_{i_1, 1, 2m_1}, \bar{h}_{i_2, 1, 2m_2+1}] = 0, \quad \label{eq54302} \\
&[\tilde {h}_{i, 0, 2m +1}, h_{j,1,r} ] = 2((\delta _{i,j} - \delta _{i,j + 1}
)\bar {h}_{i, 1, 2m + r + 1} + (\delta _{i,j} - \delta _{i, j-1}) \bar{h}_{i +1, 1,  2m + r + 1}); \quad \label{eq54560}\\
&[\tilde {h}_{i, 0, 2m}, \tilde{h}_{j, 1, 2r + 1} ] = 0; \qquad \label{eq54570}\\
&[\tilde{h}_{i,1, 2k}, \tilde{h}_{j, 0, 2l}] = 2(\delta_{i, j} - \delta_{i, j + 1}) \bar {h}_{i, 0, 2(k + l)} + 2(\delta _{i,j} - \delta _{i, j - 1})\bar {h}_{i+1, 0, 2(k + l)}; \quad \label{eq54580}\\
&[\tilde{h}_{i, 1, 2m + 1}, \tilde{h}_{j, 1, 2r} ] = 0; \qquad \label{eq54590}\\
&[x^+_{i, 1, m}, x^-_{j, 1, 2k}]  = [x^+_{i, 1, 2k}, x^{-}_{j, 1, m}]= \delta_{ij}h_{i, 0, m+2k}, \quad \label{eq5431}\\
&-[x^{+}_{i, 1, m},  x^{-}_{j, 1, 2k+1}] = [ x^{+}_{i, 1, 2k+1}, x^{-}_{j, 1, m}]= \delta_{ij}(\bar{h}_{i, 0, m+2k+1}+ \bar{h}_{i+1, 0, m+2k+1}), \quad \label{eq5432}\\
&[\tilde{h}_{i,0, 0}, x^{\pm}_{j,0,l}] = \pm \alpha_j(h_i)x^{\pm}_{j,0,l}, \quad \\
&[\tilde{h}_{i,0,0}, x^{\pm}_{j,1,l}] = \pm\alpha_j(h_i)x^{\pm}_{j,1,l}, \quad \label{eq5433}
\end{eqnarray}
\begin{eqnarray}
&[\tilde{h}_{i, 0, n+1}, x^{\pm}_{j, r, l}] = [\tilde{h}_{i, 0, n}, x^{\pm}_{j, r, l+1}] \pm \dfrac{\alpha_j(h_i)}{2} (\tilde{h}_{i, 0, n}x^{\pm}_{j,r,l} + x^{\pm}_{j,r,l}\tilde{h}_{i,0,n}), \quad \label{eq5450}\\
&[\tilde{h}_{i, 1, n+1}, x^{\pm}_{j, r, l}] = [\tilde{h}_{i, 1, n}, x^{\pm}_{j, r, l+1}] \pm \dfrac{(\delta_{i+1, j} - \delta_{i-1, j})}{2} (\tilde{h}_{i, 1, n}x^{\pm}_{j, r, l} + (-1)^r x^{\pm}_{j,r,l}\tilde{h}_{i,1,n}), \quad \label{eq5451}\\
&[x^{\pm}_{i, 1, m + 1}, x^{\pm}_{j, 0, r} ] - [x^{\pm}_{i, 1, m}, x^{\pm}_{j, 0, r + 1} ] = \pm \dfrac{(\delta_{i+1, j} - \delta_{i-1, j})}{2}(x^{\pm}_{i, 1, m} x^{\pm}_{j, 0, r} + x^{\pm}_{j, 0, r}x^{\pm}_{i, 1, m}), \quad \label{eq5453} \\
&[x^{\pm}_{i, 1, m + 1}, x^{\pm}_{j, 1, r} ] - [x^{\pm}_{i,1,m}, x^{\pm}_{j, 1, r + 1}] = \pm \dfrac{\alpha _j(h_i)}{2}(x^{\pm}_{i,1,m} x^{\pm}_{j, 1, r} - x^{\pm}_{j, 1, r}x^{\pm}_{i,1,m}), \quad \label{eq5454}\\
&[h_{i, 1, m + 1}, x^{\pm}_{j, 0, r}] - [h_{i, 1, m}, x^{\pm}_{j,0, r+ 1}] = \pm \alpha _j(h_i)/2(h_{i,1,m} x^{\pm}_{j, 0, r} + x^{\pm}_{j, 0, r}h_{i,1,m}) + \quad \nonumber \\
&(\pm \delta _{i,j + 1} - \delta _{i + 1,j} )((\bar{h}_{i, 0, m} + \bar{h}_{i+ 1, 0, m} ) x^{\pm}_{j,1,r} + x^{\pm}_{j,1,r} (\bar{h}_{i, 0, m} + \bar{h}_{i+ 1, 0, m} )), \quad \label{eq54550}\\
&\sum \limits_{\sigma \in S_2 } [x^{\pm}_{i, r, \sigma (s_1 )}, [x^{\pm} _{i, r, \sigma(s_2 )}, x^{\pm}_{j,r,s_3}]] = 0; \quad r \in \{0,1\}, i \in I, \quad \label{eq54581}\\
&[[x^{\pm}_{i_1-1, 0, t_1}, x^{\pm}_{i_1,1 , 0}],[x^{\pm}_{i_1, 1, 0}, x^{\pm}_{i_1+1,0,t_1}]] = 0, \quad t_1 \in \mathbb{Z}_+. \quad \label{eq54582}
\end{eqnarray}
We note that the relation (\ref{eq54550}) is a particular case of the relation  (\ref{eq5451}), and the relations  (\ref{eq5431}) -- (\ref{eq5432}) can be replaced by the following relation
\begin{equation}
(-1)^k[x^+_{i,1,m}, x^-_{j,1,k}] = [x^+_{i,1,k}, x^-_{j,1,m}] = \delta_{ij}(h_{i,0,m+k} - (-1)^k h_{i+1, 0, m+k}). \label{eq5434'}
\end{equation}
\end{theorem}

Let's introduce, also generating functions of generators:

\begin{eqnarray}
&h_{i,0} :=  1 + \sum_{k=0}^{\infty} h_{i,0,k} u^{-k-1}, \quad i \in I = \{1,2,\ldots, n-1\}, \\
&h_{i,1} := \sum_{k=0}^{\infty} h_{i,0,k} u^{-k-1}, \quad i \in I = \{1,2,\ldots, n-1\}, \\
&x^{\pm}_{i,j} := \sum_{k=0}^{\infty} x^{\pm}_{i,j,k} u^{-k-1} \quad, j =0,1, \quad i \in I = \{1,2,\ldots, n-1\}.
\end{eqnarray}

\vspace{2cm}

\section{Idea of isomorphism between two realization of Yangian of queer Lie superalgebra}

%In this section we construct isomorphism between Drinfel'd Yangian of the queer Lie superalgebra defined in previous section and Nazarov Yangian of the queer Lie superalgebra. This construction based on quasideterminant theory founded by I. Gelfand and  V. Retakh (\cite{G-R1}, \cite{G-R2}).

\subsection{Quasideterminants. Gelfand-Retakh theory}

I recall basic notions of the theory of quasideterminants.

We start from a definition for free skew-fields (see \cite{G-R1} ). Let $I$, $J$ be ordered sets consisting of n elements. Let $A = (a_{i,j})$, $i \in I$,  $j  \in  J$ be a matrix with formal noncommuting entries $a_{i,j}$ . Let us define by induction $n^2$ rational expression $|A|_{pq}$ of variables $a_{i,j}$, $p \in I$, $q \in J$ over a free skew-field generated by $a_{i,j}$’s. We call these expressions the quasideterminants. For $n = 1$ we set $|A|_{pq} = a_{p,q}$. For a matrix $A = (a_{i,j})$, $i \in I$,  $j  \in  J$ of order $n$ we denote $A^{\alpha, \beta}$, $\alpha \in I$,  $\beta  \in  J$ the matrix of order $n-1$ constructed by deleting the row with the index $\alpha$ and the column with the index $\beta$ in the matrix $A$. Suppose that for the given $p \in I$, $q \in J$ expressions $|A^{pq}|^{-1}_{i,j}$, $i \in I$, $j \in J$, $i \neq p$, $j \neq q$ are defined and set
$$|A|_{p,q} = a_{p,q} - \sum a_{p,j} |A^{p,q}|^{-1}_{i,j} a_{i,q}. $$

Here the sum is taken over all $i\in I \backslash \{p\}, j \in J \backslash \{q\}$.

\begin{opr}

The expression $|A|_{p,q}$ is called the quasideterminant of indices $p$ and $q$ of the matrix $A$.

\end{opr}

We will use the following important properties of quasideterminants. \\
1) The quasideterminant $|A|_{pq}$ does not depend of the permutation of rows and columns in the matrix $A$ if the $p$-th row and the $q$-th column are not changed; \\
2)  The multiplication of rows and columns. Let the matrix $B$ be constructed from the matrix $A$ by multiplication of its $i$-th row by a scalar $\lambda$ from the left. Then
$|B|_{kj} = \lambda|A|_{ij}, \quad k=i$,  $|B|_{kj} = |A|_{kj}, \quad k \neq i$ and $\lambda$ is invertible.

Let also the matrix $C$ is constructed from the matrix $A$ by multiplication of its $j$-th column by a scalar $\mu$ from the right. Then

$|C|_{il} = |A|_{ij} \mu, \quad l=j$,  $|B|_{il} = |A|_{il}, \quad l \neq j$ and $\mu$ is invertible. \\
3) The addition of rows and columns. Let the matrix $B$ is constructed by adding to some row of the matrix $A$ its $k$-th row multiplied by a scalar $\lambda$. from the left. Then $|A|_{ij} = |B|_{ij}$, $i= 1, 2, \ldots, k-1, k+1,\ldots, n$, $j = 1,2, \ldots, n$. \\
4) Let the matrix $C$ is constructed by addition to some column of the matrix $A$ its $l$-th column multiplied by a scalar $\lambda$ from the right. Then
$|A|_{ij} = |C|_{ij}$, $j= 1, 2, \ldots, l-1, l+1,\ldots, n$, $i = 1,2, \ldots, n$.

\vspace{2cm}

\subsection{Concept of definition of isomorphism }

We shall use triangular decomposition of matrix with non commuting elements and corresponding decomposition of quasideterminant.

Let
\begin{equation}
T(u) = X^{-}(u)\cdot D(u) \cdot X^{+}(u).
\end{equation}

\begin{eqnarray}
&x^-_{\alpha, \beta} = q^{1, \ldots , \alpha - 1}_{\beta, \alpha}(B^{\alpha}), \quad 1 \leq \alpha < \beta \leq n, \quad\\
& x^+_{\alpha, \beta} = q^{1, \ldots , \alpha - 1}_{\beta, \alpha}(C^{\alpha}), \quad 1 \leq \alpha < \beta \leq n, \quad\\
&d_{kk}= |A_k|_{kk}, \quad k = 1, \ldots, n.  \qquad
\end{eqnarray}

Here

\begin{eqnarray}
&A^k = (a_{ij}), \quad, i,j = k, k+1, \ldots, n, \quad\\
&B^{k} = (a_{ij}), \quad i= 1, 2, \ldots, k, \quad j = 1, \ldots , n, \quad \\
&C^{k} = (a_{ij}), \quad i= 1, 2, \ldots, n, \quad j = 1, \ldots , k \quad
\end{eqnarray}

and $q^{1, \ldots , \alpha - 1}_{\beta, \alpha}$ be a left quasi-Plucker coordinates, $r^{1, \ldots , \alpha - 1}_{\beta, \alpha}$  be a right quasi-Plucker coordinates (see \cite{GR1}.

\vspace{0.5cm}

\begin{conj}
The following map
\begin{equation}
 \Phi : Y_N(q_n) \rightarrow Y_D(sq_n)
\end{equation}
is defined by the following formulas
\begin{eqnarray}
&x^{+}_{i,0}(u) = \Phi(x^-_{i+1, i}(u)), \quad i \in I, \quad \\
&x^-_{i,0}(u) = \Phi(x^+_{i, -i-1}(u)), \quad i \in I, \quad \\
&x^{+}_{i,1}(u) = \Phi(x^-_{i+1, -i}(u)), \quad i \in I, \quad \\
&x^{-}_{i,1}(u) = \Phi(x^+_{i, i+1}(u)), \quad i\in I, \quad \\
&h_{i,0}(u) = \Phi(d^+_{i, i}(u)), \quad i\in I, \quad \\
&h_{i,1}(u) = \Phi((d_{i, -i}(u))^{-1})\Phi(d_{i+1, -i-1}(u)), \quad i\in I. \quad
\end{eqnarray}
is epimorphism.
\end{conj}

\vspace{1cm}

My research is supported by grant "Domestic Research in Moscow"' of Interdisciplinary Scientific Center J.-V. Poncelet (CNRS UMI 2615) and Scoltech Center for Advanced Study.

%\newpage

%\newpage

\vspace{2.5cm}

%The work is supported by grant "Domestic Research in Moscow"' of Interdisciplinary Scientific Center J.-V. Poncelet (CNRS UMI 2615) and Scoltech Center for Advanced Study.

\end{document}